\documentclass[11pt]{amsart}

\usepackage{amssymb,amsfonts, amsmath, amsthm} 
\usepackage[numbers]{natbib}

\newtheorem{thm}{Theorem}[section]
\newtheorem{defn}{Definition}[section]
\newtheorem{prop}{Proposition}[section]

\newtheorem{rmk}{Remark}[section]
\newtheorem{lem}{Lemma}[section]

\newtheorem{assumption}{Assumption}[section]

\DeclareMathOperator{\codim}{codim}
\DeclareMathOperator{\reg}{reg}
\DeclareMathOperator{\vol}{vol}

\DeclareMathOperator{\sing}{sing}
\DeclareMathOperator{\depth}{depth}

\usepackage{hyperref}

\begin{document}

\title{Functional Inequalities on Simple Edge Spaces}

\author{Dimitrios Oikonomopoulos}
\address{Athens, Greece}

\email{oikonomod@gmail.com, dimoiko@math.uni-bonn.de}
\subjclass[2010]{Primary MSC 46E35, MSC 53C25}

\keywords{Simple edge spaces, Functional Inequalities, Optimal Constants}

\maketitle

\begin{abstract}
In this paper we are focusing on functional inequalities on compact simple edge spaces. More precisely we address the question whether the classical functional inequalities (Sobolev, Poincar\'e) hold in this setting, and as a by-product of our methods we obtain an optimality result concerning the $B-$constant of the Sobolev inequality.
\end{abstract}




\section{Introduction}

Stratified spaces constitute an important part of singular spaces. Informally speaking, a stratified space is a topological space that can be partitioned into smooth manifolds (strata) of different dimension. Although this statement is far from complete, it is a guiding principle behind the idea of stratified spaces. The study of these spaces was initiated by Whitney \cite{Whitney}, Thom \cite{stratifiedThom} and Mather \cite{stratifiedMather} among others. Later, Goresky, MacPherson  and Cheeger studied the intersection homology and $L^2$-cohomology of these spaces (\cite{goreskyhomology1} and \cite{goreskyphersoncheeger}). It was Cheeger with his seminal paper \cite{Cheeger} that initiated the study of these spaces from an analytical point of view, and more precisely the properties of the Laplace operator on manifolds with conical singularities. The program of laying the analytic foundations of these spaces was taken up since, and still is a very active area of research.

An important role for the study of analytical questions is played by Sobolev spaces and their properties. In the case of compact stratified pseudomanifolds equipped with an iterated edge metric, Sobolev spaces have been studied as objects that describe domains of elliptic operators (\cite{hartmannleschvertmanstratified}, \cite{wittpackage}, \cite{GilKrainerMendoza},\cite{Lesch} among others) and as means to solving other problems (for example \cite{yamabestratified}, \cite{dai2017perelman}) among other research directions.

In this paper we are focusing on Sobolev spaces on compact simple edge spaces, namely stratified spaces of depth 1, equipped with an iterated edge metric. We address the question whether functional inequalities hold in this case. More precisely our aim is twofold. At first we prove the following

\begin{thm}
	Let $X$ be a compact stratified pseudomanifold of dimension $m$, endowed with an iterated edge metric on $\reg(X)$ and let $k=1$ or $k=2$. Suppose that for every singular stratum $Y$ of $X$ we have the condition
	\begin{align*}
	\codim(Y)=m-i>kp,\ \text{where}\ i=\dim(Y).        
	\end{align*}
	Then for $k=1, 2$, if $\depth(X)=1$, or the strata $Y$ with $\depth(Y) >1$ satisfy Assumption \ref{independece of y assumption}, then $X$ admits a sequence of $(k,p)-$cut-offs.	
\end{thm}
This result generalises the result obtained in \cite{parabolicityguneysubei} to second order cut-offs, and it is used as a step to prove a density theorem about Sobolev spaces, as well as it is used to provide estimates for the optimal $B$-constant in the Sobolev inequality.

After proving this result, we restrict our attention to functional inequalities on compact simple edge spaces. The reason for doing so, is that the neighborhood of a singular stratum $Y$ can be partitioned into locally Euclidean neighborhoods. This, together with a Hardy inequality allow us to obtain the Sobolev inequality for $p\in [1,m)$, with $m=\dim(X)$, i.e. there exists $A,B >0$, such that for every $u\in W^{1,p}_0(X)$ we have
\begin{equation}
\tag{$I_p$}
\|u\|_{\frac{mp}{m-p}}\leq A (\int_X |\nabla u|^pdv)^{1/p} + B(\int_X |u|^pdv)^{1/p},
\label{eqn: SobolevIneqIntroduction}
\end{equation}
where $p^* = \frac{mp}{m-p}$.

Apart from the Sobolev inequality, we prove the validity of the Rellich embedding, which states that for $X$ compact simple edge space of dimension $m>1$, with $p,q$ that satisfy $1\leq p<m$, $p\neq m-\dim(Y)$ for every singular stratum $Y$ of $X$ and $q<p^*$, the embedding
\begin{align*}
W^{1,p}_0(X)\hookrightarrow L^q(X)    
\end{align*}
is compact. Rellich embedding implies Poincar\'e inequality, i.e. for $p$ satisfying the above condition and also $p$ satisfying the condition $p<m-\dim(Y)$ for every singular stratum $Y$ of $X$, we obtain a $C>0$ such that for every $u\in W^{1,p}_0(X)$ we have
\begin{align*}
\|u-u_X\|_p \leq C \|\nabla u\|_p,
\end{align*}
where $u_X = \frac{1}{\vol(X)}\int_X u(x)dv(x)$. Combining this with the Sobolev inequality, ones obtains a Sobolev-Poincar\'e inequality, i.e. for $u\in W^{1,p}_0(X)$, we have
\begin{align*}
\|u-u_X\|_{p^*}\leq C\|\nabla u\|_p.
\end{align*}
Finally, all these constructions and chain of inequalities lead to the main theorem, namely:

\begin{thm}\label{B-optimal Constant Theorem}
	Let $X$ be a connected, compact simple edge space of dimension $m>1$. Then if $1\leq p<m-\dim(Y)$ for every singular stratum $Y$ of $X$, there exists $A>0$ such that
	\begin{equation}
	\tag{$I_{p,{B_{opt}}}$}
	\|u\|_{p^*}\leq A(\int_X |\nabla u|^pdv)^{1/p} + \vol(X)^{-\frac{1}{m}}(\int_X |u|^pdv)^{1/p}.
	\label{eqn:BestB}
	\end{equation}
	Moreover, the constant $\vol(X)^{-\frac{1}{m}}$ is optimal, in the sense that if there exists a $B>0$ such that \eqref{eqn:Sobolevp} holds with $B$, then $B\geq \vol(X)^{-\frac{1}{m}}$.
\end{thm}   

\section*{Acknowledgements}
This paper contains part of the results obtained during my PhD studies at Rheinische Friedrichs-Wilhelms-Universit\"at Bonn. I would like to thank BIGS for the financial support during the whole duration of my studies. I would also like to thank my supervisor Prof.~Dr. Matthias Lesch for his guidance and encouragement, as well as suggesting the cut-offs in Proposition 3.2.

\section{Preliminaries}

In this section we describe the class of singular spaces we are working with, and state some preliminary facts.

\begin{defn}
	A stratified space $X$ is a metrizable, locally compact, second countable space which admits a locally finite decomposition into a union of locally closed strata $\mathfrak{G}=\{Y_\alpha\}$, where each $Y_\alpha$ is a smooth, open, connected manifold, with dimension depending on the index $\alpha$. We assume the following:
	\begin{itemize}
		\item If $Y_\alpha,\ Y_\beta\in\mathfrak{G}$ and $Y_\alpha\cap\overline{Y_\beta}\neq\emptyset$, then $Y_\alpha\subseteq\overline{Y_\beta}$.
		
		\item Each stratum $Y$ is endowed with a set of 'control data' $T_Y,\ \pi_Y$ and $\rho_Y$; here $T_Y$ is a neighborhood of $Y$ in $X$ which retracts onto $Y$, $\pi_Y:T_Y\to Y$ is a fixed continuous retraction and $\rho_Y:T_Y\to [0,2)$ is a 'radial function' in this tubular neighborhood such that $\rho_Y^{-1}(0)=Y$. Furthermore, we require that if $Z\in\mathfrak{G}$ and $Z\cap T_Y\neq\emptyset$, then $(\pi_Y,\rho_Y):T_Y\cap Z\to Y\times [0,2)$, is a proper smooth submersion.
		
		\item If $W,Y,Z\in\mathfrak{G}$ and if $p\in T_Y\cap T_Z\cap W$ and $\pi_Z(p)\in T_Y\cap Z$, then $\pi_Y(\pi_Z(p)) = \pi_Y(p)$ and $\rho_Y(\pi_Z(p))=\rho_Y(p)$.
		
		\item If $Y,Z\in\mathfrak{G}$, then $Y\cap\overline{Z}\neq\emptyset\Leftrightarrow T_Y\cap Z\neq\emptyset$, $T_Y\cap T_Z\neq\emptyset \Leftrightarrow Y\subseteq \overline{Z}, \ Y=Z$ or $Z\subseteq\overline{Y}$.
		
		\item For each $Y\in\mathfrak{G}$, the restriction $\pi_Y:T_Y\to Y$ is a locally trivial fibration with fiber the cone $C(L_Y)$ over some other stratified space $L_Y$ (called the link over $Y$), with atlas $\mathcal{U}_Y=\{(\phi,\mathcal{U})\}$ where each $\phi$ is a trivialization $\pi_Y^{-1}(\mathcal{U})\to\mathcal{U}\times C(L_Y)$ and the transition functions are stratified isomorphisms of $C(L_Y)$ which preserve the rays of each conic fibre as well as the radial variable $\rho_Y$ itself, hence are suspensions of isomorphisms of each link $L_Y$ which vary smoothly with the variable $y\in\mathcal{U}$.
	\end{itemize}
	If in addition we let $X_j$ be the union of all strata of dimensions less than or equal to $j$, and require that
	\begin{align*}
	X=X_m\supseteq X_{m-1}=X_{m-2}\supseteq X_{m-3}\supseteq\dots\supseteq X_0
	\end{align*}
	and $X\setminus X_{m-2}$ is dense in $X$, then we say that $X$ is a stratified pseudomanifold of dimension $m$.
\end{defn}

The depth of a stratum $Y$ is the largest integer $k$ such that there is a chain of pairwise distinct strata $Y=Y_k,\dots,Y_0$ with $Y_j\subseteq\overline{Y_{j-1}}$ for $1\leq j\leq k$. A stratum of maximal depth is always a closed manifold. The maximal depth of any stratum in $X$ is called the depth of $X$ as a stratified space. We refer to the dense open stratum of a stratified pseudomanifold $X$ as its regular set, and the union of all other strata as the singular set,
\begin{align*}
\reg(X):= X\setminus \sing(X)\ \text{where}\ \sing(X) = \bigcup_{Y\in\mathfrak{G},\ \depth Y>0} Y.
\end{align*}

If $X$ and $X'$ are two stratified spaces, a stratified isomorphism between them is a homeomorphism $F:X\to X'$ which carries the open strata of $X$ to the open strata of $X'$ diffeomorphically and such that $\pi'_{F(Y)}\circ F = F\circ \pi_Y,\ \rho'_Y=\rho_{F(Y)}\circ F$ for all $Y\in \mathfrak{G}(X)$.

In the rest of this paper we will restrict our attention to compact stratified pseudomanifolds, which we will always denote by $X$, unless otherwise stated.

\begin{rmk}
	An implication of the notion of depth, is that a stratum $Y$ has depth $k$ if on the decomposition $V_Y\times \reg (C(L_Y))$, $L_Y$ has depth $k-1$. For example, in the case of simple edge spaces, $L_Y$ is always going to be a compact manifold.
\end{rmk}

\begin{defn}\label{iteratededgemetricdefinition}
	Let $X$ be a stratified pseudomanifold and let $g$ be a Riemannian metric on $\reg(X)$. If $\depth(X)=0$, that is $X$ is a smooth manifold, an iterated edge metric is understood to be any smooth Riemannian metric on $X$. Suppose now that $\depth(X)=k$ and that the definition of iterated edge metric is given in the case $\depth(X)\leq k-1$; then, we call a smooth Riemannian metric $g$ on $\reg(X)$ an iterated edge metric if it satisfies the following properties:
	\begin{itemize}
		\item Let $Y$ be a stratum of $X$ such that $Y\subseteq X_i\setminus X_{i-1}$. For each $q\in Y$, there exists an open neighborhood $V_Y$ of $q$ in $Y$ such that
		\begin{align*}
		\phi: \pi_Y^{-1}(V_Y)\to V_Y\times C(L_Y)
		\end{align*}
		is a stratified isomorphism. In particular,
		\begin{align*}
		\phi: \pi_Y^{-1}(V_Y)\cap \reg(X)\to V_Y\times \reg(C(L_Y))
		\end{align*}
		is a smooth diffeomorphism. Then, for each $q\in Y$, there exists one of these trivializations $(\phi, V_Y)$ such that $g$ restricted on $\pi_Y^{-1}(V_Y)\cap \reg(X)$ satisfies
		\begin{align*}
		(\phi^{-1})^*\big(g_{|_{\pi_Y^{-1}(V_Y)\cap \reg(X)}}\big) = dr^2+h_{V_Y}+r^2g_{L_Y}+k=g_0 + k,
		\end{align*}
		where $h_{V_Y}$ is the restriction on $V_Y$ of a Riemannian metric $h_Y$ defined on $Y$. $g_{L_Y}$ is a smooth family of bilinear tensors parametrized by $y\in Y$, that restricts to an iterated edge metric on $\reg(L_Y)$, and $k$ is a $(0,2)-$tensor satisfying $|k|_{g_0}=O(r^\gamma)$ for some $\gamma>0$, where $|\cdot|_{g_0}$ is the Frobenious norm.
	\end{itemize}
\end{defn}

\begin{rmk}\label{quasiisometryremark}
	The condition that $|g-g_0|_{g_0}=O(r^\gamma)$ for some $\gamma>0$ implies that $g$ and $g_0$ are quasi-isometric. That is very helpful in a variety of situations, because $g_0$ is easier to handle. For example, they produce equivalent gradient norms, i.e. $\exists\ C>0$ such that $1/C|\nabla^gu|\leq |\nabla^{g_0}u|\leq C|\nabla^gu|$. 
\end{rmk}

After introducing our setting, we recall some basic facts and definitions about Sobolev spaces on manifolds. Let $(M,g)$ be a manifold with metric $g$. We say that $f$ is equivalent to $g$ ($f\sim g$), if and only if $f(x) = g(x)$ almost everywhere with respect to the measure $\mu$ coming from the Riemannian structure. Then for $p\in [1,\infty)$, we denote by $L^p(M) = L^p$ the space of the equivalence classes of measurable functions $f:M\to \mathbb{C}$, such that
\begin{align*}
\|f\|_{L^p} = \big(\int_M |f|^p dvol_g \big)^{\frac{1}{p}}<\infty.
\end{align*}
For $p=\infty$, we define $L^\infty(M)$ as the space of the equivalence classes of measurable functions $f:M\to \mathbb{C}$, such that
\begin{align*}
\|f\|_{L^\infty} = \text{ess}\sup_M |f(x)|<\infty.
\end{align*}
For $p\in (1,\infty)$, $L^p$ is a reflexive Banach space, and for $p=2$, $L^2(M)$ is a Hilbert space with inner product
\begin{align*}
\langle f,g\rangle = \int_M f\overline{g} dvol_g.
\end{align*}
For $k\in\mathbb{N}$ and $p\in [1,\infty)$ we define
\begin{align*}
W^{k,p}(M) = \big\{ &u:M\to\mathbb{C}:\exists\ \nabla^iu\ \text{distributionally},\\ &\text{and}\ \nabla^iu\in L^p(M,T^*M^{\otimes i})\ \text{for}\ i=0,1,\dots,k\big\}
\end{align*}
with norm
\begin{align*}
\|u\|_{W^{k,p}}=\big( \sum_{i=0}^k \int_M |\nabla^iu|_{T^*M^{\otimes i}}^p dvol_g\big)^{\frac{1}{p}}.
\end{align*}
where $\nabla$ is the Levi-Civita connection induced from $g$. By adopting the Einstein summation, in local coordinates we have that 
\begin{align*}
|\nabla^iu|^p_{T^*M^{\otimes i}} :=\big(g_{T^*M^{\otimes i}} (\nabla^iu,\nabla^iu)\big)^\frac{p}{2}= \big( g^{\mu_1\nu_1}\dots g^{\mu_i\nu_i}(\nabla^iu)_{\mu_1\dots\mu_i}(\nabla^iu)_{\nu_1\dots\nu_i}\big)^{\frac{p}{2}}.
\end{align*}
For example, $(\nabla u)_\mu = \frac{\partial u}{\partial\mu}$ and $(\nabla^2 u)_{\mu\nu} = \frac{\partial^2u}{\partial\mu\partial\nu} - \Gamma^{k}_{\mu\nu}\frac{\partial u}{\partial k}$. Moreover, we define
\begin{align*}
W^{k,p}_0(M) = \overline{C_c^\infty(M)}^{\|.\|_{W^{k,p}}},
\end{align*}
i.e., the completion of smooth, compactly supported functions on $M$ with respect to the norm $\|.\|_{W^{k,p}}$.
Concerning the Sobolev space $W^{k,p}(M)$, we have the following Meyers-Serrin type theorem
\begin{prop}\label{Meyers-Serrin}
	Let $(M,g)$ be a manifold and let $W^{k,p}(M)$ be defined as above. Then the space $W^{k,p}(M)\cap C^\infty(M)$ is dense in $W^{k,p}(M)$ with respect to the norm $\|\cdot\|_{W^{k,p}}$. 			
\end{prop}
\begin{proof}
	See Theorem 2.9 in \cite{H=WGuneysu}.
\end{proof}

\section{Construction of Cut-Off Functions}\label{Construction of Cut-Off Functions}

Now let $X$ be a compact stratified pseudomanifold of arbitrary depth. In this section we show how to obtain for this space sequences of cut-off functions. We begin by giving a precise

\begin{defn}
	Let $(M,g)$ be a manifold, $p\in [1,\infty)$ and $k\in\mathbb{N}$ and let $\{\chi_n\}\subseteq C^\infty_c(M)$. We call $\{\chi_n\}$ a sequence of $(k,p)-$cut-offs if the following properties hold:
	\begin{itemize}
		\item $\forall n\in\mathbb{N}$ we have $0\leq \chi_n \leq 1$. 
		\item For every $K\subseteq M$ compact, $\exists\ n_0\in\mathbb{N}$ such that, $\forall n\geq n_0$ we have $\chi_{n}|_K=1$.
		\item $\forall j=1,\dots,k:$ $\int_M |\nabla^j\chi_n|_{T^{\otimes j}M}^p d\mu_g \to 0,$ as $n\to\infty$.
	\end{itemize}
\end{defn}

Here we will prove the existence of $(k,p)-$cut-off functions on stratified pseudomanifolds for $k=1$ and $k=2$ under some assumptions on $p$ and the iterated edge metric $g$ of $\reg(X)$. But, before doing so, we need some preliminary lemma's.

\begin{lem}\label{ChristoffelSymbols}
	Let $X$ be a stratified pseudomanifold of dimension $m$, with an iterated edge metric $g_0$, that near each singular stratum $Y$, under the trivialisation $\phi$ as in Definition \ref{iteratededgemetricdefinition}, takes the form
	\begin{align*}
	g_0 = h_{V_Y} + dr^2 + r^2 g_{L_Y}
	\end{align*}
	where $g_{L_Y}$ is a tensor parametrized by $y\in Y$ such that for each $y\in Y$ it restricts on an iterated edge metric $g_{L_Y}(y)$ on $L_Y$. Then the Christoffel symbols $\Gamma^k_{ij}$ in coordinates $r,y,z$ take the form
	\begin{itemize}
		\item For $k=r$.
		\begin{align*}
		\begin{array}{c c c}
		\Gamma^r_{rr}=0, & \Gamma^r_{ry}=0, &\Gamma^r_{rz}=0, \\
		\Gamma^r_{yy'}=0, & \Gamma^r_{yz}=0, & \Gamma^r_{zz'}=-r g_{L,zz'}.
		\end{array}
		\end{align*}
		\item For $k = z$.
		\begin{align*}
		\begin{array}{c c c}
		\Gamma^z_{rr}=0, & \Gamma^z_{ry}=0, & \Gamma^z_{rz'}=\frac{\delta^z_{z'}}{r}, \\
		\Gamma^z_{yy'}=0, & \Gamma^z_{yz'}=\frac{1}{2}\sum_{\tilde{z}}(\partial_yg_{z'\tilde{z}}g^{z\tilde{z}}), & \Gamma^z_{\tilde{z}z'}=\Gamma^{z}_{\tilde{z}z'}(g_L).
		\end{array}
		\end{align*}
		\item For $k = y$
		\begin{align*}
		\begin{array}{c c c}
		\Gamma^y_{rr}=0, & \Gamma^y_{ry'}=0, & \Gamma^y_{rz}=0, \\
		\Gamma^y_{\tilde{y}y'}=\Gamma^y_{\tilde{y}y'}(h_{V_Y}), & \Gamma^y_{y'z}=0, & \Gamma^y_{zz'}=-\frac{1}{2}\sum_{y'}(\partial_{y'}g_{zz'})h^{y'y}.
		\end{array}
		\end{align*}
	\end{itemize}
\end{lem}
\begin{proof}
	The metric $g_0$ in a local neighborhood of a singular stratum $Y$ is $g_0 = dr^2 + h_{V_Y}+r^2g_{L_Y}(y)$. Then the proposition is obtained by using the formula $\Gamma^{k}_{ij} = \frac{1}{2} \sum_l (\partial_ig_{lj}+\partial_jg_{li}-\partial_lg_{ij})g^{kl}.$
\end{proof}

We will also need the following lemma:

\begin{lem}\label{seconddifferentiallocalformula}
	Let $X$ and $g_0$ as before and let $u:\reg(X)\to\mathbb{R}$, that near a singular stratum $Y$ in local coordinates $r,y,z$, is a function of either $r$, $y$ or $z$. Then the norm of the second order covariant derivative of $u$, namely $|\nabla^2 u|_{T^{*}\otimes T^*}$ takes the form
	\begin{itemize}
		\item If $u = u(r)$, then
		\begin{align*}
		|\nabla^2 u|^2 = |\partial_r^2u|^2 + m\frac{|\partial_r u|^2}{r^2}.
		\end{align*}
		\item If $u=u(z)$, then
		\begin{align*}
		|\nabla^2 u|^2
		&= \frac{|(\nabla^{L})^2u|^2}{r^4} + 2\frac{|\nabla^{L}u|^2}{r^4} \\
		&+ 2\sum_{z,z',y,y'} g^{zz'}h^{yy'}\big[\frac{1}{2}\sum_{z_1,\tilde{z}}(\partial_yg_{z_1z})g^{\tilde{z}z_1}\frac{\partial u}{\partial\tilde{z}}\big]\big[\frac{1}{2}\sum_{z_2,\tilde{z'}}\partial_{y'}(g_{z'z_2})g^{\tilde{z'}z_2}\frac{\partial u}{\partial\tilde{z'}}\big].
		\end{align*}
		\item If $u=u(y)$, then
		\begin{align*}		
		|\nabla^2 u|^2
		& = |(\nabla^Y)^2u|^2 \\
		&+\sum_{z,\tilde{z},z',\tilde{z'}} g^{z\tilde{z}}g^{z'\tilde{z'}} \big[\sum_{y,\tilde{y}}(\partial_{\tilde{y}}g_{zz'})h^{\tilde{y}y}\frac{\partial u}{\partial y}\big]\big[\sum_{y',\tilde{y'}}(\partial_{\tilde{y'}}g_{\tilde{z}\tilde{z'}})h^{\tilde{y'}y'}\frac{\partial u}{\partial y'}\big].
		\end{align*}
	\end{itemize}
\end{lem}
\begin{proof}
	The proof makes use of the formula 
	\begin{align}\label{secondordernablanorm}
	|\nabla^2 u|^2 = \sum_{i,j,k,l} g^{ik}g^{jl}(\frac{\partial^2u}{\partial {x_i}\partial {x_j}}-\sum_{c} \Gamma^c_{ij}\frac{\partial u}{\partial {x_c}})(\frac{\partial^2u}{\partial {x_k}\partial {x_l}}-\sum_{d} \Gamma^d_{kl}\frac{\partial u}{\partial {x_d}})
	\end{align}
	where $(x_1,\dots,x_m)$ are local coordinates. Now, by using the fact that the metric $g_0$ is the direct sum of a warped product metric and another metric, we see that terms of the form $g^{ij}$ where either $i=r,j\in \{y_1,\dots,y_{\dim Y}\}$, either $i=r,j\in\{z_1,\dots,z_{\dim L}\}$ or $i\in \{y_1,\dots,y_{\dim Y}\}, j\in \{z_1,\dots,z_{\dim L}\}$ are cancelled. That allow us to consider only the cases when $i,k\in \{r,y,z\}$ and $j,l\in \{r,y,z\}$, which due to symmetry are only 6 cases. Then the proof consists of distinguishing the cases $u = u(r),u(z),u(y)$ and using Lemma \ref{ChristoffelSymbols} on the formula (\ref{secondordernablanorm}).
\end{proof}

The reason for employing Lemma \ref{seconddifferentiallocalformula} is that some constructions in this subsection will be of product type near the singular area and we would like to know how the first and second order covariant derivative behaves. A first application of this consideration allow us to obtain 

\begin{prop}\label{boundedpartitions}
	Let $X$ be a stratified pseudomanifold with metric $g_0$ and $\{U_\alpha\}_{\alpha\in A}$ an open cover. Then, there exists a subordinated partition of unity $\rho_\alpha$ such that
	\begin{itemize}
		\item supp($\rho_\alpha$)$\subseteq U_\alpha$.
		\item $\sum_{\alpha} \rho_\alpha = 1$.
		\item $\exists$ $C_\alpha>0$ such that for each $\alpha \in A$: $|\nabla \rho_\alpha|,\ |\nabla^2 \rho_\alpha| \leq C_\alpha$.
	\end{itemize}
\end{prop}
\begin{proof}See Proposition 2.2.1 in \cite{oikonomopoulos}.
\end{proof}

We have seen that on a local neighborhood of a singular stratum $Y$, where the metric takes the form $g_0 = h_Y + dr^2 + r^2g_L(y)$, an important role is played by the $y-$derivatives of the metric $g_{L,y}$. For this reason it is reasonable to form the following assumption, which we will state precisely when we use it:
\begin{assumption}\label{independece of y assumption}
	Let $X$ be a stratified pseudomanifold of dimension $m$. Let $g$ be an iterated edge metric, that near each singular stratum $Y$ takes the form $g = g_0 + k$ with $|k|_{g_0},\ |\nabla^g(k)|_{g_0} = O(r^\gamma)$ for some $\gamma>0$, and that $g_0 = h_Y + dr^2 + r^2g_{L_Y}$ with $g_{L_Y}:Y\times L_Y\to T^*L_Y\otimes T^*L_Y$ a smooth tensor. Then we assume that
	\begin{align*}
	g_{L_Y} \textit{\ is\ independent\ of\ } y\in Y.
	\end{align*}
\end{assumption}
With this assumption we obtain that $g_{L_Y}$ is constant along $y\in Y$ and therefore each $y-$derivative vanishes. The assumption $|\nabla(k)|_{g_0} = O(r^\gamma)$ gives the equivalence of the second order Sobolev space defined by $g$ and $g_0$ (see Lemma 3.4 in \cite{Pacini}). Therefore, when we consider the metric $g$ instead of $g_0$, Lemma \ref{seconddifferentiallocalformula} is true up to a constant and Proposition \ref{boundedpartitions} is true as it is. Now we are able to state the first main result of this section.
\begin{thm}\label{HessianCutOff}
	Let $X$ be a compact stratified pseudomanifold of dimension $m$, endowed with an iterated edge metric on $\reg(X)$ and let $k=1$ or $k=2$. Suppose that for every singular stratum $Y$ of $X$ we have the condition
	\begin{align}\label{codimcondition}
	\codim(Y)=m-i>kp,\ \text{where}\ i=\dim(Y).        
	\end{align}
	Then for $k=1, 2$, if $\depth(X)=1$, or the strata $Y$ with $\depth(Y) >1$ satisfy Assumption \ref{independece of y assumption}, $X$ admits a sequence of $(k,p)-$cut-offs.

\end{thm}

\begin{rmk}
	This theorem for $k=1$ is Theorem 3.4 in \cite{parabolicityguneysubei}. Here we will repeat and expand their argument, in order to cover also the case $k=2$.
\end{rmk}

Before giving the proof of this theorem, we state and prove the following proposition, which is fundamental for our construction.
\begin{prop}\label{reallinecutoff}
	Let $p\in [1,\infty)$, $k\in\mathbb{N}$ and $m-i> kp$ for some integers $m>i\geq 0$. Then, there exists a sequence of functions $\{g_n\}\subseteq C^\infty_c\big((0,2]\big)$, $n\in\mathbb{N}$ with the following properties
	\begin{itemize}
		\item $\forall n\in\mathbb{N}$ we have $0\leq g_n\leq 1$.
		\item For every $K\subset\subset (0,2]$, $\exists\ n_0\in \mathbb{N}$ such that $\forall n\geq n_0$ we have $g_{n}|_{K}=1$.
		\item For $1\leq j \leq k$ we have that $\int_0^2 |g_n^{(j)}(r)|^pr^{m-i-1}dr\to 0$, as $n\to\infty$.
	\end{itemize}
\end{prop}
\begin{proof} Let $\phi\in C^\infty_c\big((0,2]\big)$ such that $0\leq\phi\leq 1$, $\phi(r)=0$ for $0\leq r\leq 1$ and $\phi(r) = 1$ for $3/2\leq r\leq 2$. Then for $n\in\mathbb{N}$ define $g_n:(0,2]\to\mathbb{R}$ by $g_n(r) = \phi(nr)$. It is straightforward to verify the first two properties. For $1\leq j\leq k$ we have
	\begin{align*}
	\int_0^2 |g_n^{(j)}(r)|^pr^{m-i-1}dr 
	&= \int_0^2 |\phi^{(j)}(nr)|^pn^{jp}r^{m-i-1}dr\\
	&\leq C_{j} \int_{1/n}^{3/2n} n^{jp}r^{m-i-1}dr \\
	&\leq \frac{C_{j}}{m-i} \big(\big(\frac{3}{2}\big)^{m-i}-1\big) n^{jp}\big(\frac{1}{n}\big)^{m-i}.
	\end{align*}
	Since $m-i>kp\geq jp$, this converges to 0 as $n\to \infty$.

\end{proof}

Now we can give the proof of the theorem
\begin{proof}
	The proof goes by induction on the depth of the stratified pseudomanifold $X$ and a partition of unity argument. Let $X$ be a stratified pseudomanifold of depth $l\in\mathbb{N}$ and dimension $m$, $Y$ a singular stratum of dimension $i$ and $p\in Y$. Then by definition, there exists $U\subseteq X$, $V_Y\subseteq Y$ and an isometry 
	\begin{align*}
	\phi:U\to V_Y\times C(L_Y),
	\end{align*}
	where $C(L_Y)$ is the cone over a stratified pseudomanifold $L_Y$, with depth $\leq l-1$. The variables in $V_Y\times C(L_Y)$ are $y,r,z$ respectively.

	\begin{itemize}
		\item If $\depth(X) = 1$, $L_Y$ has depth 0 and it is a compact manifold without boundary. Then trivially, $b_n(y,r,z)=b_n(z)=1$ is a $(k,p)-$cut-off for $L_Y$. We then define $g_n(y,r,z)=g_n(r)$ and set $\chi_n = g_nb_n$. It is easy to see that $\{\chi_n\}$ is a $(k,p)-$cut-off in the neighborhood $V_Y\times C(L_Y)$. So the theorem, after gluing with a suitable partition of unity as done below, is proved for stratified pseudomanifolds of depth 1.
		\item Suppose now that the theorem holds for all stratified pseudomanifolds of $\depth<l$ and $\depth(X)=l$. Then $\depth(L_Y)\leq l-1$ and let $b_n(y,r,z)=b_n(z)$ be a $(k,p)-$cut-off for $L_Y$. As we can see from Proposition \ref{reallinecutoff}, for $n\in\mathbb{N},\ j=1,\dots,k:$ $|g_n^{(j)}|_\infty\leq C_{n,j}<\infty$. Set $C_n = \max_{j=1,\dots,k} C_{n,j}$. Thus w.l.o.g. we can choose $b_n$ in such a way, that for $j=1,\dots,k$ we have
		\begin{align*}
		\|\big(\nabla^{L_Y}\big)^{(j)}b_n\|_p \leq \frac{1}{nC_{n}}.
		\end{align*}
		We then set $\chi_n = g_nb_n$ and easily see that
		\begin{itemize}
			\item $0\leq\chi_n\leq 1$.
			\item For every $K\subset\subset \reg(V_Y\times C(L_Y))$, $\exists\ n_0$ such that $\forall n\geq n_0$ we have $\chi_{n}|_{K}=1$.
		\end{itemize}
		
		For $j=1,\dots,k$, we have
		
		\begin{align*}
		\|\nabla^j\chi_n\|_{L^p(U)} 
		&= \|\nabla^j(g_nb_n)\|_{L^p(U)} \\
		&\leq\sum_{a=0}^{j}C_a\|\nabla^{j-a}g_n\nabla^ab_n\|_{L^p(U)}.
		\end{align*}
		
		Since $k=1,2$ we have to check the cases $j=1,2$. For $j=1$, we obtain the terms $\|\nabla(g_n)b_n\|,\ \|g_n\nabla(b_n)\|$, and for $j=2$ the terms $\|\nabla^2(g_n)b_n\|,\ \|\nabla(g_n)\otimes\nabla(b_n)\|,\ \|g_n\nabla^2(b_n)\|$. By looking at Lemma \ref{seconddifferentiallocalformula}, we see that each of these terms is bounded by terms of the form
		\begin{align*}
		\|\frac{g^{(\alpha)}(r)}{r^{2-\alpha}}(\nabla^{L_Y})^{\beta}b_n(z)\|_{L^p},
		\end{align*}
		where $\alpha,\beta\in\{0,1,2\}$ with $1\leq\alpha+\beta\leq 2$. If $\beta=0$ the term is bounded by  
		\begin{align*}
		\big(\vol(V_Y)\vol(L_Y)\int_{\frac{1}{n}}^{\frac{3}{2n}} n^{\alpha p}r^{m-i-1-(2-\alpha)p}dr\big)^{1/p}\to 0 \ \text{as\ } n\to\infty,
		\end{align*}
		since $m-i>2p$. If $\beta>0$, then the term is bounded by
		\begin{align*}
		\vol(V_Y)^{1/p} |g_n^{(\alpha)}|_{\infty} \big(\int_0^2 r^{m-i-1-(2-\alpha)p}dr\big)^{1/p} \|\big(\nabla^{L_Y}\big)^{\beta}b_n\|_{L^p(L_Y)}\to 0
		\end{align*}	
		as $n\to\infty$, since $m-i>2p\geq (2-\alpha)p$, and $g_n^{(\alpha)}\leq C_n$.
	\end{itemize}
	Thus, $\{\chi_n\} =\{g_nb_n\}$ is a sequence of $(k,p)-$cut-offs in the neighborhood $U=V_Y\times C(L_Y)$.		
	Finally, the sequence of functions $\{\tilde{\chi_n}\}$ defined by
	
	\begin{align*}
	\tilde{\chi}_n = \sum_{i} \rho_i\chi_{n,U_i} + \rho_{int},
	\end{align*}
	is a sequence of $(k,p)-$cut-offs. In order to see this, we take $1\leq j\leq k$ and we calculate:
	\begin{align*}
	\|\nabla^j \tilde{\chi}_n\|_{L^p}
	\leq& \|\sum_{i}\nabla^j(\rho_i) \chi_{n,U_i} + \nabla^j\rho_{int}\|_{L^p} \\
	+ & \|\sum_{i,\ a+\beta\leq j,\ \beta>0} C_{a,\beta} \nabla^a(\rho_i)\nabla^\beta(\chi_{n,U_i})\|_{L^p}.
	\end{align*}
	In view of Proposition \ref{boundedpartitions}, the first term converges locally uniformly to $0$ as $n\to\infty$ and it is bounded since $\sum_{i}\nabla^j(\rho_i) \chi_{n,U_i} + \nabla^j\rho_{int}$ is bounded and $L^p$-integrable. Thus by Lebesgue's theorem we conclude that it converges to $0$. For the second term, we use that $\chi_{n,U_i}$ is a $(k,p)-$cut-off, and that $\nabla^a \rho_i$ is bounded, thus it also converges to $0$.
\end{proof}

\subsection{A Density Theorem for $W^{1,p}(X)$}
Now, by using the construction of the weak cut-off functions, we are able to prove a density result about the Sobolev spaces $W^{1,p}(\reg(X))$. Before stating the precise result we prove the following intermediate proposition:

\begin{prop}\label{boundedapproximation in W^1p}
	Let $(M,g)$ be a Riemannian manifold and $p\in [1,\infty)$. Then the space $\{u\in W^{1,p}(M):\ \|u\|_\infty <\infty\}$ is dense in $W^{1,p}(M)$ in the $\|\cdot\|_{W^{1,p}}$-norm.
\end{prop}
\begin{proof}
	Let $u\in W^{1,p}(M)$. By Proposition \ref{Meyers-Serrin} we can assume that $u\in C^\infty(M)$. Then, let $a_n \to \infty$ be regular values of $u$ and define $u_n = \max (-a_n,\min(u,a_n))$. Then $u_n \in W^{1,p}(M)$ and it is easily seen that $\|u_n-u\|_{W^{1,p}}\to 0$ as $n\to\infty$.
\end{proof}
Therefore, we have the following Theorem:

\begin{thm}\label{densitytheorem}
	Let $X$ be a compact stratified pseudomanifold of dimension $m$, endowed with an iterated edge metric on $\reg(X)$. Suppose that for every singular stratum $Y$ of $X$ we have the condition
	\begin{align}\label{codimcondition}
	\codim(Y)=m-i>p,\ \text{where}\ i=\dim(Y).        
	\end{align}
	Then, 
	\begin{align*}
	W^{1,p}(\reg(X)) = W^{1,p}_0(\reg(X)).
	\end{align*}
\end{thm}

\begin{proof}
	The inclusion $W^{1,p}_0(\reg(X))\subseteq W^{1,p}(\reg(X))$ is obvious. For the converse, let $u\in W^{1,p}(\reg(X))$. In virtue of Proposition \ref{boundedapproximation in W^1p}, we can assume that $u \in  L^\infty(X)\cap W^{1,p}(\reg(X))$. Since $m-i\geq p$, we obtain from Theorem \ref{HessianCutOff}, a sequence $\{\chi_n\}$ of $(1,p)-$cut-offs. We set $u_n = \chi_n u$ and we calculate:
	\begin{align*}
	\|u-u_n\|_{W^{1,p}} 
	&= \|u-u_n\|_{L^p} + \|\nabla(u)-\chi_n\nabla(u)-\nabla(\chi_n)u\|_{L^p} \\ 
	&\leq \|u-u_n\|_{L^p} + \|\nabla(u)-\chi_n\nabla(u)\|_{L^p}+\|\nabla(\chi_n)u\|_{L^p}.
	\end{align*}
	The first two terms converge to $0$ by Lebesgue's Theorem, and the last term is bounded by $\|\nabla(\chi_n)\|_{L^p} \cdot \|u\|_{L^\infty}\to 0$ as $n\to\infty$ since $\chi_n$ is a $(1,p)-$cut-off and $u$ is bounded. We see that $u_n\in W^{1,p}(\reg(X))$ and has compact support in $\reg(X)$. Therefore approximating with functions in $C^\infty_c(\reg(X))$ (for example by mollifying) concludes the proof.
\end{proof}

\section{A Hardy Inequality}

In this section, we prove a Hardy-type inequality for simple edge spaces. Hardy inequality playes a crucial role in proving the validity of the Sobolev embedding which we will show in the next section. We begin, by proving a weighted Hardy inequality on the real half-line. More precisely, we have: 

\begin{prop}\label{HardyInequality}
	Let $p\geq1$, $f\in \mathbb{N}$ with $p\neq f+1$. Then for every $u\in C^\infty_c((0,+\infty))$ we have
	\begin{align}
	\int_0^\infty \frac{|u(r)|^p}{r^p} r^{f}dr \leq \bigg|\frac{p}{f+1-p}\bigg|^p \int_0^\infty |\partial_ru|^pr^{f}dr.
	\end{align}
\end{prop}
\begin{proof}
	We have
	\begin{align*}
	\int_0^\infty \frac{|u|^p}{r^p} r^{f}dr = 
	&\int_0^\infty |u|^p\frac{(r^{f+1-p})'}{f+1-p}dr\\
	&=-\frac{p}{f+1-p}\int_0^\infty |u|^{p-1}\text{sgn}(u)(\partial_ru)r^{f+1-p}dr \\
	& \leq\bigg|\frac{p}{f+1-p}\bigg|\int_0^\infty |u|^{p-1}|\partial_ru|r^{f+1-p}dr.
	\end{align*}
	For $p=1$ the statement has been proved. For $p>1$ we split $r^{f+1-p} =  r^{\frac{f(p-1)}{p}+(1-p)}\cdot r^{\frac{f}{p}}$
	and apply H{\"o}lder Inequality with $\frac{p-1}{p}+\frac{1}{p}=1$. Then we obtain
	\begin{align*}
	\int_0^\infty \frac{|u|^p}{r^p} r^{f}dr \leq \frac{p}{|f+1-p|} \bigg(\int_0^\infty \frac{|u|^p}{r^p} r^{f}dr\bigg)^{1-\frac{1}{p}}  \bigg(\int_0^\infty |\partial_ru|^pr^{f}dr\bigg)^{\frac{1}{p}}.
	\end{align*}
	Taking powers of $p$ finishes the proof.
\end{proof}
Using the above proposition, we easily obtain a Hardy-type inequality for model simple edge spaces. More precisely we have
\begin{prop}\label{HardyCone}
	Let $(L,g_L)$, $(Y,h)$ be manifolds of dimension $f,\ d$ respectively, $1\leq p<\infty$ with $p\neq f+1$ and let $Y\times C(L)=Y\times(0,+\infty)\times L$ with metric $g_0=h+dr^2+r^2g_L$, where $g_L:Y\times L\to T^*L\otimes T^*L$ is a smooth tensor that restricts to a Riemannian metric on each $y\in Y$. Then for $u\in C^\infty_c(Y\times C(L))$ one has 
	\begin{align*}
	\int_{Y\times C(L)} \frac{|u|^p}{r^p}dvol_{g_0} \leq\bigg|\frac{p}{f+1-p}\bigg|^p \int_{Y\times C(L)} |\nabla^{g_0} u|^pdvol_{g_0}.
	\end{align*}
\end{prop}
\begin{proof}
	The volume form on $Y\times C(L)$ is $r^fdrdvol_hdvol_{g_L}$ and for simplicity we denote it by $r^fdrdydz$. Thus, we have
	\begin{align*}
	\int_{Y\times C(L)} \frac{|u(r,y,z)|^p}{r^p}r^fdrdy&dz 
	= \int_L\int_Y \int_0^\infty \frac{|u(r,y,z)|^p}{r^p} r^fdrdydz\\
	&\leq\bigg|\frac{p}{f+1-p}\bigg|^p \int_L\int_Y \int_0^\infty |\partial_ru(r,y,z)|^p r^fdrdydz\\
	&\leq \bigg|\frac{p}{f+1-p}\bigg|^p  \int_L\int_Y \int_0^\infty |\nabla^{g_0} u(r,y,z)|^p r^fdrdydz\\
	& = \bigg|\frac{p}{f+1-p}\bigg|^p\int_L\int_Y \int_0^\infty |\nabla^{g_0} u(r,y,z)|^p dvol_{g_0}.
	\end{align*}
	where in the first inequality we applied Proposition \ref{HardyInequality} and in the second inequality the fact that $|\partial_ru|^p\leq |\nabla^{g_0}u|^p = \big(|\partial_ru|^2+\frac{|\nabla^{g_L}u|^2}{r^2}+|\nabla^hu|^2\big)^{p/2}$.
\end{proof}

\section{Geometry of Simple Edge Spaces}\label{Geometry Of Simple Edge Spaces Section}

In this section, we explore in more detail the geometry of simple edge spaces. As we will see, we can choose a finite cover of any singular stratum $Y$ of $X$, such that on the regular part near $Y$, an iterated edge metric is equivalent to the Euclidean. We will use this in the next section in order to obtain a Sobolev inequality on simple edge spaces. We first begin with a Lemma:
\begin{lem}\label{MintoS^n}
	Let $(L,g_L)$ be a compact manifold without boundary, with $\dim(L)=n$. Then, for every $\varepsilon>0$, there exists a finite cover of $L$ by charts $(U_i,\phi_i)$, such that each $U_i$ can be embedded into $S^n$, through a map $f_i:U_i\to f(U_i)\subseteq S^n$, and on $U_i$ we have
	\begin{align}
	(1-\varepsilon)f_i^*(g_{S^n})\leq g_L\leq (1+\varepsilon)f_i^*(g_{S^n})
	\end{align}
	as bilinear forms, and where $S^n$ is the standard $n$-sphere.
\end{lem}
\begin{proof}
	Let $1>\varepsilon>0$, $p\in L$ and let $(U_p,\phi)$ to be normal coordinates around $p$. We denote them by $x_1(p),\dots,x_n(p)$. By shrinking $U$, we can assume that $\phi :U\to B(0,\delta)$ is a diffeomorphism for each $1>\delta>0$. If $v = \sum_i a_i\partial_{x_i}$ in local coordinates, then we have that
	\begin{align}\label{L with epsilon}
	(1-\varepsilon)\sum_i a_i^2 \leq g_L(v,v)\leq (1+\varepsilon)\sum_i a_i^2.
	\end{align}
	Since $\delta>0$ is arbitrary, we can assume that $\delta<1$. Then we consider the map $f:B(0,\delta)\to S^n$ which is defined by
	\begin{align*} 
	f(x_1,\dots,x_n)=(x_1,\dots,x_n,\sqrt{1-(x_1^2+\dots+x_n^2)}).
	\end{align*}
	Denote by $\partial_{x_i}^L, \partial_{x_i}^{S^n}$ the vectors fields with regard to these two maps in $L$ and $S^n$ respectively. Then these are related by $f_*(\phi_*(\partial_{x_i}^L))=\partial_{x_i}^{S^n}$. A simple computation shows that
	\begin{align}\label{metric on sphere}
	g_{S^n}(\partial_{x_i}^{S^n},\partial_{x_j}^{S^n}) = \delta_{ij} + \frac{x_ix_j}{(1-(x_1^2+\dots+x^2_n))},
	\end{align}
	since $g_{S^n} = i^*(g_{\mathbb{R}^{n+1}})$. By taking $v = \sum_i a_i\partial_{x_i}^{S^n}$ and making $\delta>0$ small enough, by unfolding definitions and using (\ref{metric on sphere}) we obtain that
	\begin{align}\label{S^n with epsilon}
	(1-\varepsilon)\sum_i a_i^2 \leq g_{S^n}(v,v)\leq (1+\varepsilon)\sum_i a_i^2.
	\end{align}
	Combining (\ref{L with epsilon}) and (\ref{S^n with epsilon}) and setting $f_i = f\circ \phi$, we obtain that
	\begin{align*}
	(1-\varepsilon)f_i^*(g_{S^n})\leq g_L\leq (1+\varepsilon)f_i^*(g_{S^n}).
	\end{align*}
	Since $L$ is compact, we can find a finite cover of $L$ by $(U_i,\phi_i)$ that each embeds into $S^n$ through an $f_i$ and that concludes the proof.
\end{proof}

\begin{lem}\label{lemmaCone}
	Let $(L,g_L)$ be a compact manifold with $\dim(L)=n-1$ and for $a>0$ consider the straight cone of $L$, i.e. the manifold $C_a(L) = \big((0,a)\times L, dr^2+r^2g_L\big)$. Then, for every $\varepsilon>0$, there exists a finite cover of $C_a(L)$ with charts $(V_i,\psi_i)$ such that each $V_i$ embeds into $\mathbb{R}^n$ through $f_i$ and on each $V_i$ we have
	\begin{align*}
	(1-\varepsilon)f_i^*(\delta_{kl})\leq dr^2+r^2g_L\leq (1+\varepsilon)f_i^*(\delta_{kl}),
	\end{align*}
	as bilinear forms, and where $\delta_{kl}$ stands for the Euclidean metric $g_{\mathbb{R}^n}(\partial_{x_k},\partial_{x_l})$.	
\end{lem}
\begin{proof}
	Take $1>\varepsilon>0$. By Lemma \ref{MintoS^n} we obtain a finite cover of $L$ by $(U_i,\phi_i)$ such that each $U_i$ embeds to $S^n$ through an $\tilde{f}_i$. This yields the cover $(V_i,\psi_i)$ for $C_a(L)$ which is defined by $V_i = C_a(U_i)$, $\psi_i = (id, \phi_i)$. Then by defining $F_i=(id, \tilde{f}_i)$ we see that $C_a(U_i)$ embeds into $C_a(S^n)$ which we identify with $B(0,a)\setminus\{0\} \subseteq \mathbb{R}^n$ through polar coordinates $\lambda : C_a(S^n)\to B(0,a)\setminus \{0\}$. By setting $f_i = \lambda\circ F_i$, we obtain that
	\begin{align*}
	(1-\varepsilon)f_i^*(\delta_{kl})\leq dr^2+r^2g_L\leq (1+\varepsilon)f_i^*(\delta_{kl}),
	\end{align*}
	and that concludes the proof.
\end{proof}

\begin{lem}\label{lemmaEdge}
	Let $(Y,h)$ be a compact Riemannian manifold without boundary with $\dim(Y) = b$, and $L$ be a compact manifold with $\dim(L)=n$. Let, $g_L:Y\times L\to T^*L\times T^*L$, be a smooth tensor, such that for every $y\in Y$, it restricts to a Riemannian metric on $L$. Then, for every $\varepsilon>0$ there exists a finite covering with charts $\big(W_i,\psi_i\big)$, of 
	\begin{align*}
	\bigg(Y\times(0,a)\times L, h+ dr^2 + r^2 g_L\bigg),
	\end{align*}
	such that each $W_i$ embeds into $\mathbb{R}^{1+b+n}$ through a map $f_i:W_i\to f_i(W_i)\subseteq \mathbb{R}^{1+b+n}$ and on $W_i$ we have
	\begin{align*}
	(1-\varepsilon)f_i^*(\delta_{kl})\leq h+dr^2+r^2g_L\leq(1+\varepsilon)f_i^*(\delta_{kl}),
	\end{align*}
	as bilinear forms, where $\delta_{kl}$ stands for the Euclidean metric $g_{\mathbb{R}^{1+b+n}}(\partial_{x_k},\partial_{x_l})$.
\end{lem}
\begin{proof}
	Let $\varepsilon>0$ and fix $y_0\in Y$. Since $L$ is compact, by Lemma \ref{lemmaCone}, we can find a finite cover $\big(U_i,\ \phi_i\big)$ of $L$, such on $U_i$ we have
	\begin{align}\label{g_y_0}
	(1-\varepsilon)\delta_{kl}\leq g_{L,y_0}\leq (1+\varepsilon)\delta_{kl}	
	\end{align}
	as bilinear forms, where $\delta_{kl}$ stands for the Euclidean metric. Let now $y\in V_{y_0}\subseteq Y$, $V_{y_0}$ open and let $v\in T_{U_i}L$ with $v = \sum_k a_i\frac{\partial}{\partial x_i}$, where $x_1,\cdots,x_n$ are the coordinates on $\phi_i(U_i)$. Then we have $g_{L,y}(v,v) = g_{L,y_0}(v,v) + g_{L,y}(v,v)-g_{L,y_0}(v,v)$. We estimate 
	\begin{align}\label{g_y-g_y_0}
	|g_{L,y}(v,v)-g_{L,y_0}(v,v)| 
	&= |\sum_{k,l} \big(g_{{L,y}_{kl}}-g_{{L,y_0}_{kl}}\big)a_ka_l|\nonumber\\ 
	&\leq \sum_{k,l} |\big(g_{{L,y}_{kl}}-g_{{L,y_0}_{kl}}\big)|\frac{a_k^2+a_l^2}{2} \nonumber\\
	&\leq \frac{\varepsilon}{2n}\sum_{k,l} \big(a_k^2+a_l^2\big)\nonumber\\
	&= \frac{\varepsilon}{2n} 2n |v|^2 = \varepsilon\sum_k a_k^2.
	\end{align}
	for $y\in V_{y_0}'\subseteq V_{y_0}$, since $g_{L}:Y\times L\to T^*L\otimes T^*L$ is a smooth tensor. Now, we are in position to apply Lemma \ref{MintoS^n} and Lemma \ref{lemmaCone} on the family $g_{L,y}$, with $y\in V'_{y_0}$. Then, we obtain that $C_a(U_i)$ embeds into $C_a(S^n)$ through $F_i$, which is diffeomorphic to $B(0,a)\setminus\{0\}$ under polar coordinates $\lambda$, and that for $y\in V_{y_0}'$, $v\in T_{U_i}$, we have
	\begin{align*}
	(1-2\varepsilon)(\lambda\circ F_i)^*(\delta_{kl})\leq dr^2+r^2g_{L,y}\leq (1+2\varepsilon)(\lambda\circ F_i)^*(\delta_{kl})	
	\end{align*}
	as bilinear forms, where $\delta_{kl}$ stands for the Euclidean metric. Then, by taking a possibly smaller neighborhood $V''_{y_0}$ of $y_0\in Y$, and normal coordinates $\phi_{y_0}$ at $y_0$, we obtain that
	\begin{align*}
	(1-\varepsilon)\delta_{kl}\leq h\leq (1+\varepsilon)\delta_{kl}.
	\end{align*}
	as bilinear forms. Then by taking the charts $\big( V''_{y_0}\times U_{i,y_0}, \ \phi_{y_0}\times \phi_i\big)_{i=1}^{N_{y_0}}$ we see that $V_{y_0}''\times C_a(U_{i,y_0})$ embeds to $\phi_{y_0}(V_{y_0}'')\times (B(0,a)\setminus\{0\})\subseteq \mathbb{R}^{1+b+n}$ through $f_i:= \phi_{y_0}\times (\lambda\circ F_i)$ and on $V_{y_0}''\times C_a(U_{i,y_0})$
	\begin{align*}
	(1-\varepsilon)f_i^*(\delta_{kl})\leq h+dr^2+r^2g_L\leq(1+\varepsilon)f_i^*(\delta_{kl})
	\end{align*}
	as bilinear forms. Since $Y$ is compact, if we repeat the above procedure for each $y\in Y$, we can find a finite family 
	\begin{align*}
	\{ V_{y_j}\times U_{i,j}, \ \phi_{y_j}\times f_{i,j}\}_{j=1,\cdots,N, \ i=1,\cdots N_{j}}.
	\end{align*}
	such that on each cover $V_{y_i}\times U_{i,j}$ we have
	\begin{align*}
	(1-\varepsilon)f_{i,j}^*(\delta_{kl})\leq h+dr^2+r^2g_L\leq(1+\varepsilon)f_{i,j}^*(\delta_{kl})
	\end{align*}
	as bilinear forms, where $\delta_{kl}$ stands for the Euclidean metric. That concludes the proof.
\end{proof}

On a simple edge space, the metric has the form
\begin{align}
g = g_0 +k,
\end{align}
where $|k|_{g_0} =  O(r^\gamma)$ for $\gamma>0$, $r$ the radial variable of the cone, and $g_0 = h+dr^2+r^2g_L$. Since $g$ and $g_0$ are quasi isometric, Lemma \ref{lemmaEdge} applies to $g$, but not necessarily with constants $1-\varepsilon,\ 1+\varepsilon$. In order to obtain constants like these, one should restrict to small neighborhoods around the singular strata.

Now let $X$ be a compact simple edge space. For simplicity we assume that it has only one stratum $Y$ of depth 1. (In the case where we have more than one singular strata, we proceed in the same way). As a consequence, there exists a neighborhood $U\subseteq X$, compact manifold $L$ and a locally trivial fibration
\begin{align*}
\phi: U\to Y\times C_2(L),
\end{align*}
such that
\begin{align*}
(\phi^{-1})^*(g)_{|_U} = g_0 + k,
\end{align*}
where $g_0=h+dr^2+r^2g$ and $|k|_{g_0}=O(r^\gamma)$ for some $\gamma>0$. Since $Y,\ L$ are compact, one can find finite covers $U_j$, $V_{i(j),j}$ respectively, and thus $\phi^{-1}\big(U_j\times C_2(V_{i,j})\big)$ is an open cover for $U$. According to the previous considerations and Remark \ref{quasiisometryremark}, one can choose this open cover such that each open set of this cover embeds into $\mathbb{R}^m$ through an $f$ and the metric there is equivalent with the Euclidean, i.e.
\begin{align}\label{Euclidianmetric}
\frac{1}{4} f^*(\delta_{ij}) \leq g \leq 4f^*(\delta_{ij}).
\end{align}
From now on, whenever we refer to this cover, we will just write $U_j\times C_2(V_i)$ instead of $U_j\times C_2(V_{i(j),j})$. Since $X$ is compact, on then can find a finite cover $M_\lambda$ such that $X\setminus U\subseteq \bigcup_\lambda M_\lambda$. Moreover, one can choose this cover so that (\ref{Euclidianmetric}) holds. Define now the projection
\begin{equation}
\begin{split}
&\pi : Y\times C_2(L)\to Y\times L,\\
&\pi(y,r,z)=(y,z).
\end{split}
\end{equation}
Then if $\chi_i,\ \psi_j$ are partitions of unity associated to the cover $U_i,\ V_j$ respectively, then $\pi^* (\chi_i\psi_j)$ is a partition of unity, associated with the cover $U_i\times C_2(V_j)$. For simplicity, we set $\rho_{ij}(y,r,z)= \pi^*\big(\chi_i(y)\psi_j(z)\big)$. An important observation is that for $u\in C^\infty(U_i\times C_2(V_j))$ we have
\begin{align*}
|\nabla^{g_0}u|^2 = |\partial_ru|^2+\frac{|\nabla^{g_L}u|^2}{r^2} + |\nabla^hu|^2,
\end{align*}
and thus, for $\rho_{ij}$ we have the bound
\begin{align}\label{partition1/r}
|\nabla^{g_0}\rho_{ij}|\leq \frac{C}{r}.
\end{align}

\begin{rmk}
	A novel difference between the partitions of unity $\rho_{ij}$ we considered here, and the partitions of unity we considered in Proposition \ref{boundedpartitions} is that the former are not bounded. The reason why this happens is that we defined them in open subsets $V_j\subseteq L$. But the partitions of unity in Proposition \ref{boundedpartitions} are independent of the $z$-variable, for $z\in L$, therefore they are bounded. 
\end{rmk}

\section[Functional Inequalities on Simple Edge Spaces]{Functional Inequalities on Simple Edge\\ Spaces}\label{FunctionalInequalitiesSection}

This, together with the next section constitute the core part of the manuscript. Here we are establishing functional inequalities on compact simple edge spaces. To do so, we heavily rely on the identification of the singular neighborhood near a singular stratum with the Euclidean as shown in section \ref{Geometry Of Simple Edge Spaces Section}. We note here that in our case the differential of the partitions of unity are not bounded. For this reason we frequently employ the version of Hardy inequality which we proved before (see Proposition \ref{HardyCone}). In the following part, $C$ will denote a generic constant that may vary from line to line. 

To begin with, we note that the Sobolev inequality holds in this context:

\begin{prop}(Sobolev Embedding)\label{SobolevEmbedding}
	Suppose $X$ is a compact simple edge space of dimension $m>1$. Then, there exists $A,\ B>0$ such that for all $u\in C^\infty_c(\reg(X))$ we have
	\begin{align}
	\big(\int_X |u|^{\frac{m}{m-1}}dvol_{g}\big)^{\frac{m-1}{m}} \leq A\int_X |\nabla u|dvol_{g} + B\int_X |u|dvol_{g}.
	\end{align}
\end{prop}
\begin{proof}
	We show the proof in the case where we have only one singular stratum $Y\subseteq U$. In the case where we have more, we can apply the same procedure in every stratum. Take $\phi_1:\reg(X)\to[0,1]$, with $\text{supp}(\phi_1)\subseteq U$ and $\phi_1 = 1$ for $r\leq 1$, $\phi_1 =0$ for $r\geq 3/2$. Set $\phi_2 = 1-\phi_1$. Then we have
	\begin{align}\label{phidecomposition}
	\|u\|_{\frac{m}{m-1}} \leq \|\phi_1u\|_{\frac{m}{m-1}}+\|\phi_2u\|_{\frac{m}{m-1}}.
	\end{align}
	Concerning $\phi_1u$ we have
	\begin{align*}
	\|\phi_1u\|_{\frac{m}{m-1}} 
	&= \|\sum_{ij} \rho_{ij}\phi_1u\|_{\frac{m}{m-1}}\leq \sum_{ij}\| \rho_{ij}\phi_1u\|_{\frac{m}{m-1}}.
	\end{align*}
	Recall, that each neighborhood $U_{ij}= V_i\times C_2(U_j)$ can be embedded through an embedding $f$ by Lemma \ref{lemmaEdge} to $\mathbb{R}^b\times C_2(S^n)$ which we identify with a subset of $\mathbb{R}^{m}$ through cylindrical coordinates $\lambda$. Therefore we obtain
	\begin{align*}
	\|\rho_{ij}\phi_1u\|_{\frac{m}{m-1}} 
	&= \bigg(\int_{V_i\times C_2(U_j)} |\rho_{ij}\phi_1u|^{\frac{m}{m-1}}dvol(g_0)\bigg)^{\frac{m-1}{m}}\\
	&\leq \bigg(C\int_{V_i\times C_2(U_j)} |\rho_{ij}\phi_1u|^{\frac{m}{m-1}} (\lambda\circ f)^*(dx)\bigg)^{\frac{m-1}{m}}\\
	&= \bigg(C\int_{(\lambda\circ f)(V_i\times C_2(U_j))} |(\rho_{ij}\phi_1u)\circ (\lambda\circ f)^{-1}|^{\frac{m}{m-1}} dx\bigg)^{\frac{m-1}{m}}\\
	&\leq C\int_{(\lambda\circ f)(V_i\times C_2(U_j))} |\nabla (\rho_{ij}\phi_1u)\circ (\lambda\circ f)^{-1}|dx\\
	& = C\int_{V_i\times C_2(U_j)} |\nabla (\rho_{ij}\phi_1u)|(\lambda\circ f)^*(dx)\\
	&\leq C \int_{V_i\times C_2(U_j)} |\nabla (\rho_{ij}\phi_1u)| dvol(g_0),
	\end{align*}
	where on the first and the third inequality we used Lemma \ref{lemmaEdge} and on the second inequality the Sobolev inequality on $\mathbb{R}^m$ with $m = 1+b+n$, since $(\lambda\circ f)(V_i\times C_2(U_j))\subseteq \mathbb{R}^m$. By (\ref{partition1/r}) and Proposition \ref{HardyCone}, we obtain that the former is bounded by
	\begin{align}\label{SobEdge}
	C_1 \int_{U_i\times C_2(V_j)}|\nabla (\phi_1u)| dvol_{g_0}\leq C_1 \int_X |\nabla u|dvol_{g_0} + C_2\int_X|u|dvol_{g_0}. 
	\end{align}
	Concerning $\phi_2u$ one procceeds exactly as above with respect to a cover of $X\setminus U$. The difference is that the partitions of unity are uniformly bounded, thus combining it with (\ref{SobEdge}) one obtains the required result.
\end{proof}

By classical means (see \cite{HebeySobolev} Lemma 3.1), one also obtains that for $1\leq p<m$,
\begin{align}\label{Sobolevnp/n-p}
W^{1,p}_0 (X)\hookrightarrow L^{\frac{mp}{m-p}}(X)
\end{align}
continuously. Furthermore, by using Theorem \ref{densitytheorem}, we obtain that if $1\leq p< m-\dim(Y)$ for every singular stratum $Y$ of $X$ and $p<m$, then
\begin{align*}
W^{1,p}(X) = W_0^{1,p}(X)\hookrightarrow L^{\frac{mp}{m-p}}(X).
\end{align*}
\begin{rmk}
	The Sobolev inequality with exponent $p=2$ holds in the more general case of compact stratified pseudomanifolds. For a proof, using different methods, see \cite{yamabestratified}, \cite{mondelo}.
\end{rmk} 
Moreover in this setting, the classical Rellich-Kondrachov theorem is true.
\begin{prop}\label{rellich simple edge}
	Let $X$ be a compact simple edge space of dimension $m>1$, and let $p,q$ satisfy $1\leq p<m$, $p\neq m-\dim(Y)$ for every singular stratum $Y$ of $X$, and $q<p^*=\frac{mp}{m-p}$. Then the embedding
	\begin{align*}
	W^{1,p}_0 (X)\hookrightarrow L^q(X)
	\end{align*}
	is compact.
\end{prop}

\begin{proof}
	Let $\{u_n\}_{n\in\mathbb{N}}\subseteq W^{1,p}_0 (X)$ such that
	\begin{align*}
	\|u_n\|_{W^{1,p}_0}\leq M<\infty.
	\end{align*}
	Since $u_n\in W^{1,p}_0$, there exists $\tilde{u}_n \in C^\infty_c(\reg(X))$ such that $\|u_n-\tilde{u}_n\|_{W^{1,p}}<\frac{1}{n}$. Notice that $\|\tilde{u}_n\|_{W^{1,p}}\leq M+1$. If we find a convergent subsequence of $\{\tilde{u}_n\}$ that converges to $v\in L^q$, which we denote again by $\tilde{u}_n$, then $u_n\to v$ in $L^q$, because $\|u_n-v\|_{L^q} \leq \|u_n-\tilde{u}_n\|_{L^q} + \|\tilde{u}_n-v\|_{L^q}$. Notice that the second terms converges to $0$, and for the first term we have
	\begin{align*}
	\|u_n-\tilde{u}_n\|_{L^q} 
	&\leq C\|u_n-\tilde{u}_n\|_{L^{p^*}} \\ 
	&\leq C \|u_n-\tilde{u}_n\|_{W^{1,p}}\\
	&\leq \frac{C}{n}\to 0.
	\end{align*}
	Therefore, we can assume that $\{u_n\}\subseteq C^\infty_c(\reg(X))$. As before, we can find covers $U_i\times C_2(V_j)$ of $\cup_{Y\in \sing(X)}Y$, such that each cover is embedded though an embedding $f$ into $\mathbb{R}^b\times C_2(S^n)$ which we identify with a subset of $\mathbb{R}^{m}$ through $\lambda$. Then, as in the proof of Proposition \ref{SobolevEmbedding} we have
	\begin{align*}
	&\|(\rho_{ij}\phi_1u_n)\circ(\lambda\circ f)^{-1}\|^p_{W^{1,p}_{B(0,1)}} \\
	&= \int_{B(0,1)} |\nabla(\rho_{ij}\phi_1u_n)\circ(\lambda\circ f)^{-1}|^pdx + \int_{B(0,1)} |(\rho_{ij}\phi_1u_n)\circ(\lambda\circ f)^{-1}|^pdx\\
	&\leq C\bigg(\int_{(\lambda\circ f)^{-1}(B(0,1))} |\nabla(\rho_{ij}\phi_1u_n)|^pdvol_{g_0} + \int_{(\lambda\circ f)^{-1}(B(0,1))} |(\rho_{ij}\phi_1u_n)|^pdvol_{g_0}\bigg)\\
	& \leq C \bigg(\int_{(\lambda\circ f)^{-1}(B(0,1))} |\frac{\phi_1u_n}{r}|^pdvol_{g_0} + \| \phi_1u_n\|^p_{W^{1,p}(X)}\bigg)\\
	&\leq C \|u_n\|^p_{W^{1,p}(X)}\leq C(M+1),
	\end{align*}
	where on the first inequality we used Lemma \ref{lemmaEdge}, on the second we used that $|\nabla(\rho_{ij})\leq \frac{C}{r}$ and on the last we used Hardy inequality, i.e. Proposition \ref{HardyCone}.
	Then by Rellich-Kondrachov theorem, for $q<\frac{mp}{m-p}$, in each cover $(\lambda\circ f)(U_i\times C_2(V_j))\subseteq B(0,1)$ we find a convergent subsequence of $u_n$ in $L^q$. For the interior part we apply the classical Rellich-Kondrachov theorem and thus, after passing to a subsequence we obtain the desired result.
\end{proof}

\begin{rmk}
	See also Theorem 6.1 in \cite{beiBochner}
\end{rmk}
\begin{rmk}
	We can utilize Theorem \ref{densitytheorem} again. Under the hypotheses of Proposition \ref{rellich simple edge} and by assuming furthermore that $1\leq p < m-\dim(Y)$ for every singular stratum $Y$ of $X$, we obtain that for $1\leq q< p^*$, the embedding
	\begin{align*}
	W^{1,p}(X) = W^{1,p}_0(X)\hookrightarrow L^q(X)
	\end{align*}
	is compact.
\end{rmk}
From now on, we assume that we have the condition $1\leq p < m -\dim(Y)$ for every singular stratum $Y$ of $X$. So, there is no distinction between $W^{1,p}(X)$ and $W_0^{1,p}(X)$. With this condition, by using Proposition \ref{rellich simple edge} and the fact that a compact simple edge space has finite volume, one can prove the following version of Poincar\'e Inequality
\begin{prop}(Poincar\'e Inequality) Let $X$ be a connected, compact simple edge space of dimension $m>1$ and let $1\leq p<m$, $p< m-\dim(Y)$ for every singular stratum $Y$ of $X$. Then there exists a constant $C>0$ such that for $u\in W^{1,p}(X)$ we have
	\begin{align}\label{poincare}
	\|u-u_X\|_{p}\leq C \|\nabla u\|_p,
	\end{align}
	where $u_X = \frac{1}{\vol(X)}\int_Xu(x)dvol_{g}$.
\end{prop}
\begin{proof}
	Suppose that (\ref{poincare}) is not true. Then, for every $k\in\mathbb{N}$ there exists $u_k\in W^{1,p}(X)$ such that 
	\begin{align*}
	\|u_k-(u_k)_X\|_{p}> k \|\nabla u_k\|_p.
	\end{align*}
	Set $v_k=\frac{u_k-(u_k)_X}{\|u_k-(u_k)_X\|_{p}}\in W^{1,p}(X)$. Then by hypothesis we have that
	\begin{align*}
	\|\nabla v_k\|_p<\frac{1}{k},\ \ \|v_k\|_p=1.
	\end{align*}
	Since $\|v_k\|_{W^{1,p}}$ is uniformly bounded, we can apply Proposition \ref{rellich simple edge} and obtain a subsequence, which we denote again by $v_k$, that converges strongly in $L^p(X)$, since $p<p^*$. Thus there exists $v\in L^p(X)$ such that $v_k\to v$ strongly in $L^p(X)$. Then we have that $\|v\|_p=1$, $v_X=0$ and $\|\nabla v_k\|_p\to 0$. Pairing $v$ against a test function $\phi\in C^\infty_c(X)$ gives
	\begin{align*}
	-\int_X v\nabla \phi 
	&= -\lim_{k\to\infty}\int_X v_k\nabla \phi \\
	&= \lim_{k\to\infty}\int_X \nabla v_k \phi\\
	&=0.
	\end{align*}
	That is $\nabla v=0$, thus $v$ is constant and since $v_X=0$ then it is $0$, which is a contradiction.
\end{proof}
Using the Sobolev Embedding, one can prove a stronger version of Poincar\'e inequality, namely the Sobolev-Poincar\'e inequality:
\begin{prop}(Sobolev-Poincar\'e Inequality)
	Let $X$ be a connected, compact simple edge space of dimension $m>1$ and let $1\leq p< m$, $p<m-\dim(Y)$ for every singular stratum $Y$ of $X$. Then there exists a constant $C>0$ such that for $u\in W^{1,p}(X)$ we have   
	\begin{align}\label{Sobolevpoincare}
	\|u-u_X\|_{p^*}\leq C \|\nabla u\|_p,
	\end{align}
	where $u_X = \frac{1}{\vol(X)}\int_Xu(x)dvol_{g}$.
\end{prop}
\begin{proof}
	The proof is a combination of Sobolev Embedding (\ref{Sobolevnp/n-p}) and Poincar\'e Inequality. For details, see \cite{HebeySobolev} Proposition 3.9.
\end{proof}

\section{Optimization of Constants}\label{B-Constant Section}

In this section we are focusing on obtaining optimal constants of Sobolev inequalities. In order to do so, we use the cut-off functions introduced in section \ref{Construction of Cut-Off Functions}. We obtain optimal results concerning the constants of the $L^p$ norms of the functions in the embeddings $W^{2,p}_0\hookrightarrow W^{1,p^*}_0$ and $W^{1,p}_0\hookrightarrow L^{p^*}$. To be more precise, in the previous section we proved the Sobolev embedding on a compact simple edge space $X$, i.e. 

\begin{equation}
\tag{$I_p$}
\|u\|_{p^*}\leq A\|\nabla u\|_p+B\|u\|_p
\label{eqn:Sobolevp}
\end{equation}
with $u\in W_0^{1,p}(X)$ and $1\leq p<\dim(X)=m$.

\subsection{The Embedding $W^{1,p}_0\hookrightarrow L^{p^*}$.}

The construction of $(1,p)-$cut-off functions allow us to prove some optimal results concerning the constant $B>0$. To be more precise, we set	

\begin{align*}
B_p(X) = \inf\{B>0: \text{such\ that}\ \exists \ A>0\ \text{such\ that}\ \eqref{eqn:Sobolevp} \  \text{holds}\}.
\end{align*}
Two questions that are of interest are
\begin{itemize}{}
	\item Compute $B_p(X)$.
	\item Does there exist an $A>0$ such that \eqref{eqn:Sobolevp} holds with $B=B_p(X)$?
\end{itemize}
This questions are part of the so called AB-programm, which consists of finding the optimal constants for various functional inequalities, such as the Sobolev inequality and Sobolev-Poincar\'e inequality (for more details, see \cite{HebeySobolev} and \cite{ABProgramm}). In this section we answer these questions in the case of compact simple edge spaces with the condition $1\leq p< m-\dim(Y)$ for every singular stratum $Y$ of $X$. For this reason, we can apply Theorem \ref{HessianCutOff}. The condition $1\leq p < m-\dim(Y)$ for every singular stratum $Y$ of $X$, guarantees the existence of a sequence of cut-off functions $\{\chi_n\}_{n\in\mathbb{N}}\subseteq C^\infty_c(\reg(X))$ such that
\begin{itemize}
	\item $0\leq \chi_n\leq 1$.
	\item $\forall$ compact $K\subseteq \reg(X),\ \exists\ n_0\in\mathbb{N}$ such that $\forall\ n\geq n_0$ we have $\chi_{n_{|K}}=1$.
	\item $\int_X |\nabla \chi_n|^p dv_{g}\to 0$.
\end{itemize}
Now, plug $\chi_n$ in \eqref{eqn:Sobolevp}. Since $vol(X)<\infty$, using the properties of these cut-offs and Lebesgue's dominated convergence theorem, we obtain by taking $n\to \infty$
\begin{align*}
\vol(X)^{\frac{1}{p^*}}\leq B\vol(X)^{\frac{1}{p}},
\end{align*}
which gives a lower bound for $B$, i.e.
\begin{align}\label{Bbound}
B\geq \vol(X)^{-\frac{1}{m}}.
\end{align}
This gives that $\forall \ p $ with $1\leq p\leq m-\dim(Y)$ for every singular stratum $Y$ of $X$, we have $\vol(X)^{-\frac{1}{m}}\leq B_p(X)$. Now using a Sobolev-Poincar\'e inequality and the fact that $\vol(X)<\infty$, one can see that $B_p(X)$ is attainable. More precisely, by Sobolev-Poincar\'e we have that for $1\leq p<m$, $p< m-\dim(Y)$ for every singular stratum $Y$ of $X$, there exists $C>0$ such that
\begin{align*}
\|u-u_X\|_{p^*}\leq C\|\nabla u\|_p,
\end{align*}
for $u\in W^{1,p}_0(X)$. Using triangle inequality we obtain that
\begin{align*}
\|u\|_{p^*}
&\leq C\|\nabla u\|_p + \|u_X\|_{p^*}\\
&\leq C\|\nabla u\|_p + \vol (X)^{\frac{1}{p^*}-1}\int_X |u|\\
&\leq C\|\nabla u\|_p + \vol(X)^{\frac{1}{p^*}-1}\|u\|_p \vol(X)^{1-\frac{1}{p}}\\
&= C\|\nabla u\|_p + \vol(X)^{-\frac{1}{m}}\|u\|_p.
\end{align*}
Combining this with the lower bound (\ref{Bbound}) we obtain the following Theorem.

\begin{thm}\label{B-optimal Constant Theorem}
	Let $X$ be a connected, compact simple edge space of dimension $m>1$. Then if $1\leq p<m-\dim(Y)$ for every singular stratum $Y$ of $X$, there exists $A>0$ such that for every $u\in W^{1,p}(X)$ we have
	\begin{equation}
	\tag{$I_{p,{B_{opt}}}$}
	\|u\|_{p^*}\leq A\|\nabla u\|_p + \vol(X)^{-\frac{1}{m}}\|u\|_p.
	\label{eqn:BestB}
	\end{equation}
	Moreover, the constant $\vol(X)^{-\frac{1}{m}}$ is optimal, in the sense that if there exists a $B>0$ such that \eqref{eqn:Sobolevp} holds with $B$, then $B\geq \vol(X)^{-\frac{1}{m}}$.
\end{thm}

\subsection{The Embedding $W^{2,p}_0\hookrightarrow W^{1,p^*}_0$}

Now we focus on the embedding $W^{2,p}\hookrightarrow W^{1,p^*}$. Recall, that this embedding is obtained by the embedding $W^{1,p}\hookrightarrow L^{p^*}$ and the Kato inequality
\begin{align*}
|\nabla |\nabla^k u||\leq |\nabla^{k+1}u|,
\end{align*}
for $u\in C^\infty(\reg(X))$. Therefore we obtain positive constants $A,B,C>0$, such that for every $u\in C^\infty_c(\reg(X))$ we have:
\begin{align}\label{Second Order Sobolev ineq}
\|\nabla u\|_{L^{p^*}} + \|u\|_{L^{p^*}} \leq A \|\nabla^2u\|_{L^{p}} + B\|\nabla u\|_{L^p} + C\|u\|_{L^p}.
\end{align}
with $1\leq p < m$. Now, by Theorem \ref{HessianCutOff}, the condition $2p^* <m-\dim(Y)$ for every singular stratum $Y$, implies the existence of $(2,p^*)-$cut-offs. Notice now by using H\"older inequality and $\vol(X)<\infty$, that if we have a sequence $\{\chi_n\}_{n\in\mathbb{N}}$ of $(2,p^*)-$cut-offs, then this is a sequence of $(1,q)$ and $(2,q)-$cut-offs, for $q\leq p^*$. Thus, by plugging it in (\ref{Second Order Sobolev ineq}) and letting $n\to\infty$, we obtain
\begin{align*}
C\geq \vol(X)^{-\frac{1}{m}}.
\end{align*}
Similarly as before, we apply (\ref{Sobolevpoincare}) and Kato inequality on the function $v = |\nabla u|$, since $v\in W^{1,p}$ and $1\leq p < p^*<m-\dim(Y)$ for every singular stratum $Y$ of $X$ and we obtain
\begin{align*}
\||\nabla u|-|\nabla u|_X\|_{p^*}
&\leq C \|\nabla^2u\|_p.
\end{align*}
On the other hand we have
\begin{align*}
\||\nabla u|_X\|_{p^*} 
&= \vol(X)^{\frac{1}{p^*}-1} \int |\nabla u|\\
&\leq \vol(X)^{\frac{1}{p^*}-1} \|\nabla u\|_p \vol(X)^{1-\frac{1}{p}}\\
&= \vol(X)^{-\frac{1}{m}}\|\nabla u\|_p.
\end{align*}
Finally, we use Theorem \ref{B-optimal Constant Theorem} and by adding the inequalities we obtain
\begin{align*}
\|\nabla u\|_{p^*} + \|u\|_{p^*} \leq C\|\nabla^2u\|_p + (A+\vol(X)^{-\frac{1}{m}}) \|\nabla u\|_p + \vol(X)^{-\frac{1}{m}}\|u\|_p.
\end{align*}
Therefore, we have the following 
\begin{thm}\label{Second Order Optimal Sobolev Theorem}
	Let $X$ be a connected, compact simple edge space of dimension $m>1$. Then if $p\in [1,\infty)$, with $1\leq 2p^*<m-\dim(Y)$ for every singular stratum $Y$ of $X$, then there exists $A,B>0$ such that
	\begin{equation}
	\tag{$I_{p,2,Opt}$}
	\|\nabla u\|_{p^*} + \|u\|_{p^*} \leq A\|\nabla^2u\|_p + B \|\nabla u\|_p + \vol(X)^{-\frac{1}{m}}\|u\|_p.
	\label{eqref:SecondOrderIp}
	\end{equation}
	Moreover, the constant $\vol(X)^{-\frac{1}{m}}$ is optimal, in the sense that if there exists a $C>0$ such that (\ref{Second Order Sobolev ineq}) holds with $C>0$, then $C\geq \vol(X)^{-\frac{1}{m}}$.
\end{thm}

\nocite{*}
\bibliographystyle{plain}

\bibliography{preprintBib}

\begin{thebibliography}{10}

\bibitem{yamabestratified}
Kazuo Akutagawa, Gilles Carron, and Rafe Mazzeo.
\newblock The {Y}amabe problem on stratified spaces.
\newblock {\em Geom. Funct. Anal.}, 24(4):1039--1079, 2014.

\bibitem{wittpackage}
Pierre Albin, \'{E}ric Leichtnam, Rafe Mazzeo, and Paolo Piazza.
\newblock The signature package on {W}itt spaces.
\newblock {\em Ann. Sci. \'{E}c. Norm. Sup\'{e}r. (4)}, 45(2):241--310, 2012.

\bibitem{beiBochner}
Francesco Bei.
\newblock Sobolev spaces and {B}ochner {L}aplacian on complex projective
  varieties and stratified pseudomanifolds.
\newblock {\em J. Geom. Anal.}, 27(1):746--796, 2017.

\bibitem{parabolicityguneysubei}
Francesco Bei and Batu G\"{u}neysu.
\newblock {$q$}-parabolicity of stratified pseudomanifolds and other singular
  spaces.
\newblock {\em Ann. Global Anal. Geom.}, 51(3):267--286, 2017.

\bibitem{stratifiedBrasselet}
J.-P. Brasselet, G.~Hector, and M.~Saralegi.
\newblock Th\'{e}or\`eme de de {R}ham pour les vari\'{e}t\'{e}s
  stratifi\'{e}es.
\newblock {\em Ann. Global Anal. Geom.}, 9(3):211--243, 1991.

\bibitem{Cheeger}
Jeff Cheeger.
\newblock Spectral geometry of singular {R}iemannian spaces.
\newblock {\em J. Differential Geom.}, 18(4):575--657 (1984), 1983.

\bibitem{goreskyphersoncheeger}
Jeff Cheeger, Mark Goresky, and Robert MacPherson.
\newblock {$L^{2}$}-cohomology and intersection homology of singular algebraic
  varieties.
\newblock In {\em Seminar on {D}ifferential {G}eometry}, volume 102 of {\em
  Ann. of Math. Stud.}, pages 303--340. Princeton Univ. Press, Princeton, N.J.,
  1982.

\bibitem{dai2017perelman}
Xianzhe Dai and Changliang Wang.
\newblock Perelman's $ w $-functional on manifolds with conical singularities.
\newblock {\em arXiv preprint arXiv:1711.08443}, 2017.

\bibitem{ABProgramm}
Olivier Druet and Emmanuel Hebey.
\newblock The {$AB$} program in geometric analysis: sharp {S}obolev
  inequalities and related problems.
\newblock {\em Mem. Amer. Math. Soc.}, 160(761):viii+98, 2002.

\bibitem{GilKrainerMendoza}
Juan~B. Gil, Thomas Krainer, and Gerardo~A. Mendoza.
\newblock On the closure of elliptic wedge operators.
\newblock {\em J. Geom. Anal.}, 23(4):2035--2062, 2013.

\bibitem{goreskyhomology1}
Mark Goresky and Robert MacPherson.
\newblock Intersection homology theory.
\newblock {\em Topology}, 19(2):135--162, 1980.

\bibitem{MR1944550}
Daniel Grieser and Matthias Lesch.
\newblock On the {$L^2$}-{S}tokes theorem and {H}odge theory for singular
  algebraic varieties.
\newblock {\em Math. Nachr.}, 246/247:68--82, 2002.

\bibitem{H=WGuneysu}
Davide Guidetti, Batu G\"{u}neysu, and Diego Pallara.
\newblock {$L^1$}-elliptic regularity and {$H=W$} on the whole {$L^p$}-scale on
  arbitrary manifolds.
\newblock {\em Ann. Acad. Sci. Fenn. Math.}, 42(1):497--521, 2017.

\bibitem{hartmannleschvertmanstratified}
Luiz Hartmann, Matthias Lesch, and Boris Vertman.
\newblock On the domain of {D}irac and {L}aplace type operators on stratified
  spaces.
\newblock {\em J. Spectr. Theory}, 8(4):1295--1348, 2018.

\bibitem{HebeySobolev}
Emmanuel Hebey.
\newblock {\em Sobolev spaces on {R}iemannian manifolds}, volume 1635 of {\em
  Lecture Notes in Mathematics}.
\newblock Springer-Verlag, Berlin, 1996.

\bibitem{Lesch}
Matthias Lesch.
\newblock {\em Operators of {F}uchs type, conical singularities, and asymptotic
  methods}, volume 136 of {\em Teubner-Texte zur Mathematik [Teubner Texts in
  Mathematics]}.
\newblock B. G. Teubner Verlagsgesellschaft mbH, Stuttgart, 1997.

\bibitem{stratifiedMather}
John~N. Mather.
\newblock Stratifications and mappings.
\newblock pages 195--232, 1973.

\bibitem{melrose}
Richard~B. Melrose.
\newblock {\em The {A}tiyah-{P}atodi-{S}inger index theorem}, volume~4 of {\em
  Research Notes in Mathematics}.
\newblock A K Peters, Ltd., Wellesley, MA, 1993.

\bibitem{mondelo}
Ilaria Mondello.
\newblock The local {Y}amabe constant of {E}instein stratified spaces.
\newblock {\em Ann. Inst. H. Poincar\'{e} Anal. Non Lin\'{e}aire},
  34(1):249--275, 2017.

\bibitem{oikonomopoulos}
Dimitrios Oikonomopoulos.
\newblock Functional inequalities and heat kernel asymptotics on some classes
  of singular riemannian manifolds.
\newblock {\em PhD Thesis, Rheinische Friedrich-Willhelms-Universit\"at Bonn},
  2019.

\bibitem{Pacini}
Tommaso Pacini.
\newblock Desingularizing isolated conical singularities: uniform estimates via
  weighted {S}obolev spaces.
\newblock {\em Comm. Anal. Geom.}, 21(1):105--170, 2013.

\bibitem{Petersen}
Peter Petersen.
\newblock {\em Riemannian geometry}, volume 171 of {\em Graduate Texts in
  Mathematics}.
\newblock Springer, New York, second edition, 2006.

\bibitem{stratifiedThom}
R.~Thom.
\newblock Ensembles et morphismes stratifi\'{e}s.
\newblock {\em Bull. Amer. Math. Soc.}, 75:240--284, 1969.

\bibitem{Whitney}
Hassler Whitney.
\newblock Complexes of manifolds.
\newblock {\em Proc. Nat. Acad. Sci. U. S. A.}, 33:10--11, 1947.

\bibitem{YoussinCohomology}
Boris Youssin.
\newblock {$L^p$} cohomology of cones and horns.
\newblock {\em J. Differential Geom.}, 39(3):559--603, 1994.

\end{thebibliography}

\end{document}